\documentclass[12pt,a4paper]{article}

\usepackage[T2A]{fontenc}
\usepackage[utf8]{inputenc}
\usepackage[russian, english]{babel}

\usepackage{amssymb,amsthm,amsmath,amsfonts,ccaption,array,epic,eepic}
\usepackage{setspace}
\usepackage{tikz}
\usetikzlibrary{arrows}
\usepackage{ifthen}
\usepackage{arrayjobx}

\ifx\LibLexLoaded\undefined\else\expandafter \fi
\let\LibLexLoaded\relax

\def\makeatletter{\catcode`\@=11\relax}

\begingroup \makeatletter
\gdef\atc@end{\catcode`\@\errmessage{Unmatched \string\restoreatcode}}
\gdef\atc@stack{\atc@end}
\gdef\atc@next#1\atc@end{\def\atc@stack{#1\atc@end}}
\gdef\saveatcode{%
 \expandafter\expandafter\expandafter\atc@next
 \expandafter\the\expandafter\catcode\expandafter`\expandafter\@
 \expandafter\space\expandafter\atc@next\atc@stack
}
\gdef\restoreatcode{\catcode`\@=\atc@stack}
\endgroup

\saveatcode\makeatletter

\def\tempcount{\lib@alloc0\count\countdef\insc@unt}
\def\tempdimen{\lib@alloc1\dimen\dimendef\insc@unt}
\def\tempskip{\lib@alloc2\skip\skipdef\insc@unt}
\def\tempmuskip{\lib@alloc3\muskip\muskipdef\@cclvi}
\def\tempbox{\lib@alloc4\box\chardef\insc@unt}
\def\temptoks{\lib@alloc5\toks\toksdef\@cclvi}
\def\tempread{\lib@alloc6\read\chardef\sixt@@n}
\def\tempwrite{\lib@alloc7\write\chardef\sixt@@n}

\def\lib@alloc#1#2#3#4#5{%
  \advance\count1#1by\@ne
  \ifnum\count1#1<#4\else\lib@alloc@err#2\fi
  #3#5=\count1#1\relax
}
\def\lib@alloc@err#1{\errmessage{No room for a new temporary #1}}

\let\lib@tmp\undefined
\let\lib@next\undefined

\def\lib@empty{} 

\let\lib@gtmp\undefined 

\newif\ifyes

\def\testfile#1{%
  \begingroup
  \tempread\lib@tmp
  \openin\lib@tmp=#1\relax
  \ifeof\lib@tmp
    \closein\lib@tmp \endgroup \yesfalse
  \else
    \closein\lib@tmp \endgroup \yestrue
  \fi
}

\restoreatcode

\saveatcode\makeatletter

\newtoks\evp@shadow 
\evp@shadow{\evp@hidden}
\def\evp@hidden{\evp@hidden}

\def\@atnextpar{}
\def\@ateverypar{}

\newif\if@evp 

\def\evp@setup{%
 \if@evp\else
  \@evptrue
  \evp@shadow\everypar
  \everypar{%
   \let\lib@tmp\@atnextpar \global\let\@atnextpar\lib@empty \lib@tmp
   \@ateverypar
   \the\evp@shadow
  }%
  \let\everypar\evp@shadow
  \aftergroup\evp@aftergroup
 \fi
}

\def\evp@aftergroup{%
 \ifx\@atnextpar\lib@empty
  \begingroup
   \edef\lib@tmp{\the\evp@shadow}%
   \ifx\lib@tmp\evp@hidden \else
    \global\everypar\evp@shadow
    \global\evp@shadow{\evp@hidden}%
   \fi
  \endgroup
 \else
  \evp@setup
 \fi
}

\def\atnextpar#1{%
 \evp@setup
 \begingroup
  \expandafter\toks@\expandafter{\@atnextpar #1}%
  \xdef\@atnextpar{\the\toks@}%
 \endgroup
}

\def\ateverypar#1{%
 \evp@setup
 \begingroup
  \expandafter\toks@\expandafter{\@ateverypar #1}%
  \xdef\lib@gtmp{\the\toks@}%
 \endgroup
 \let\@ateverypar\lib@gtmp
 \let\lib@gtmp\undefined
}

\restoreatcode

\saveatcode \makeatletter

\newbox\risbox
\def\setrisbox{\setbox\risbox}
\def\riswidth{\wd\risbox}

\def\rismove(#1,#2){%
 \setrisbox\vbox{%
  \dimen@=#2\relax
  \vskip\dimen@
  \hbox{\dimen@=#1\relax\hskip\dimen@\box\risbox\hskip-\dimen@}%
  \vskip-\dimen@
 }%
}

\let\global@ris@options\lib@empty

\def\resetrisoptions{\let\global@ris@options\lib@empty}

\def\addrisoption#1{%
  \expandafter\toks@\expandafter{\global@ris@options #1\relax}%
  \edef\global@ris@options{\the\toks@}%
}

\def\ris@options{} 

\def\atnextris#1{%
  \expandafter\toks@\expandafter{\ris@options #1\relax}%
  \edef\ris@options{\the\toks@}%
}

\newif\ifrisleft
\newdimen\risspace
\newtoks\risadjust
\newtoks\risadjustabove
\newtoks\risadjustbelow
\newdimen\risminheight
\newcount\risshowbox

\newdimen\ris@totalspace
\newdimen\ris@hsize

\def\rischeckpage{%
  \dimen@\ht\risbox
  \advance\dimen@\dp\risbox
  \ifdim\dimen@<\risminheight \dimen@\risminheight \fi
  \advance\dimen@\pagetotal
  \advance\dimen@\pagedepth
  \advance\dimen@ -\pageshrink
  \ifdim\dimen@>\pagegoal
    \ifrisleft \risleftfalse \else \rislefttrue \fi
  \fi
}

\def\ris{\begingroup
  \risleftfalse
  \risspace 1em
  \risadjust{}%
  \risadjustabove{}%
  \risadjustbelow{}%
  \risminheight\z@
  \risshowbox\z@
  \global@ris@options
  \ris@options
  \futurelet\lib@next\ris@checkopt
}

\def\ris@parseopt{%
  \futurelet\lib@next\ris@checkopt
}

\def\ris@checkopt{%
 \expandafter\if\space\noexpand\lib@next
   \expandafter\ris@checkopt@space
  \else
   \expandafter\ris@checkopt@
  \fi
}
\def\ris@checkopt@space{\afterassignment\ris@parseopt\let\lib@next= }

\def\ris@checkopt@{%
 \ifx\lib@next[\expandafter\ris@getopt\else\expandafter\ris@\fi
}

\def\ris@getopt[#1]{#1\ris@}

\def\ris@{%
  \ris@breakpar{\the\risadjustabove\relax\the\risadjust\relax}%
  \risadjust{}
  \ris@totalspace\riswidth
  \advance\ris@totalspace\risspace
  \ris@hsize\hsize
  \advance\ris@hsize-\ris@totalspace
  \ifrisleft
    \global\parshape 1 \ris@totalspace \ris@hsize
    \ateverypar{\global\parshape 1 \ris@totalspace \ris@hsize}%
  \else
    \global\parshape 1 \z@ \ris@hsize
    \ateverypar{\global\parshape 1 \z@ \ris@hsize}%
  \fi
  \ifhmode \ris@insert \else \atnextpar{\ris@insert}\fi
  \no@page@breaks
  \ignorespaces
}

\def\ris@breakpar#1{%
  \ifhmode{%
   \parfillskip\z@\parskip\z@\tolerance\@m\par #1\parindent\z@\leavevmode
  }\else #1\fi
}

\def\no@page@breaks{%
  \ifx\tex@par@npb\undefined
    \let\tex@par@npb\par
    \let\tex@penalty@npb\penalty
    \interlinepenalty\@M
    \predisplaypenalty\@M
    \postdisplaypenalty\@M
    \def\par{\tex@par@npb\tex@penalty@npb\@M\relax}%
    \def\penalty{\ifvmode\count@\else\tex@penalty@npb\fi}%
  \fi
}

\def\end@no@page@breaks{%
  \ifx\tex@par@npb\undefined \else
    \let\par\tex@par@npb
    \let\penalty\tex@penalty@npb
    \let\tex@par@npb\undefined
    \let\tex@penalty@npb\undefined
  \fi
}

\def\ris@insert{{%
 \ifx\ris@totalspace\undefined
  \errmessage{Side insertion with no text}
 \else
  \ifnum\risshowbox>0
    \setrisbox\vbox{%
      \vskip-.4pt\hrule 
      \hbox{\hskip-.4pt\vrule \box\risbox \vrule\hskip-.4pt}%
      \hrule\vskip-.4pt
    }%
  \fi
  \setbox\z@\lastbox
  \hbox to\z@{%
   \ifrisleft \kern-\ris@totalspace
   \else \kern\ris@hsize \kern\risspace \fi
   \vbox to\z@{\kern -1ex \box\risbox\vss}\hss
  }%
  \box\z@
 \fi
}}

\def\endris{\end@no@page@breaks\futurelet\lib@next\endris@}

\def\endris@{%
 \ifx\lib@next\par
   \par \the\risadjustbelow\relax\the\risadjust\relax
 \else
   \ris@breakpar{\the\risadjustbelow\relax\the\risadjust\relax}%
 \fi
 \endgroup
 \global\parshape\z@
 \let\ris@options\lib@empty 
}

\def\epsrisbox#1{\hbox{%
  \testfile{#1}\ifyes
    \def\lib@tmp{\epsfbox{#1}}%
  \else
    \message{^^JWarning: file #1 not found^^J}%
    \global\epsfxsize\z@
    \global\epsfysize\z@
    \let\lib@tmp\relax
  \fi
  \lib@tmp
}}

\restoreatcode

\addrisoption{%
  \risminheight 3\baselineskip
}

\addrisoption{%
  \pretolerance=100 \tolerance=2000
  \hyphenpenalty=10 \exhyphenpenalty=10
}

\usepackage[text={17.5cm,25cm}, includehead]{geometry}

\def\gplus#1{\mathop{+\!\!_{_#1}}}

\def\ZZ{{\mathbb Z}}

\theoremstyle{plain}
\newtheorem{theorem}{Theorem}         

\newtheorem{lemma}[theorem]{Lemma}
\newtheorem{corollary}{Corollary}[theorem]
\theoremstyle{definition}
\newtheorem*{primer}{Example}
\theoremstyle{remark}

\captiondelim{. }

\begin{document}
\pagestyle{myheadings}

\makeatletter
\renewcommand{\@oddhead}{\hfil For which graphs the sages can guess correctly  the color of at least one hat. 
\hfil\thepage}
\makeatother

\title{For which graphs the sages can guess \\ correctly the color of at least one hat}
\author{K.~Kokhas\thanks{St.Petersburg State University, St.Petersburg, Russia}, 
A.~Latyshev\thanks{ITMO University, St.Petersburg, Russia}}

\maketitle


\begin{abstract}
Several sages wearing colored hats occupy the vertices of a graph.
Each sage tries to guess the color of his own hat merely on the basis of observing the hats
of his neighbours without exchanging any information. Each hat
can have one of \emph{three} colors. A predetermined guessing strategy is winning
if it guarantees at least one correct individual guess for every assignment of
colors. We completely solve the problem of describing graphs  for which the sages win.
\end{abstract}

\subsection{Introduction}

We consider a variant of the well-known hat guessing game, which we call <<Hats>> for short.

In this variant, there are $n$ sages and a lot of hats of $k$ different colors.
The sages sit at the vertices of an undirected graph $G$. Two sages see each
other if they occupy adjacent vertices. The sages participate in the following
TEST. The adversary gives them hats in such way that everybody can see his
neighbour's hats, but can not see his own hat and does not know its color.
Any communication between sages is prohibited. The sages try to guess the colors
of their own hats simultaneously (for example, on command from the adversary). We say
that sages pass the test (successfully) or win if at least one of them
guessed the color correctly.

Before the test, the sages are informed about its rules, and then they
have a session at which they should determine a public deterministic strategy
(public in the sense that everybody, including the adversary, knows it). 
A strategy is said to be \emph{winning} if for all possible placements of hats, 
at least one of the sages guesses the color of his hat correctly. 
We say that the sages win if they have a
winning strategy, and lose otherwise.

The main question of this paper is for which graphs the sages win. 
We consider the game <<Hats>> for $k=3$ only.

The <<Hats>> game was mentioned for the first time in 1961 in M.\,Gardner's
book~\cite{Gardner}. In 1998, T.\,Ebert considered it in his PhD
thesis~\cite{Ebert}, and the game became very popular. After that, a lot of variants
of <<Hats>>, mainly probabilistic, have been discovered. For example, 36 variants
of the rules are listed in~\cite{Krzywkowski}.

Let us list some known results for the <<Hats>> game. 

\begin{lemma}
  \label{subgraph}
  Let $G'$ be a subgraph of a graph $G$. If the sages have a winning strategy on $G'$, 
	then they also win on $G$.
\end{lemma}
It is obvious.

\begin{theorem}[\cite{BHMKL}]
  \label{tree}
  If G is a tree, then the sages lose.
\end{theorem}

\begin{theorem}[\cite{BHMKL}]
  \label{thr:two-colors-in-tree}
  Let $G$ be a tree and $A$ be any vertex of $G$. Then there exist two colors
  $c_1$ and $c_2$ and two losing placements of hats such that in the first
  placement of hats sage $A$ gets a~hat of color $c_1$, and in the second
  placement of hats sage $A$ gets a~hat of color $c_2$.
\end{theorem}

W.\,Szczechla~\cite{Szczechla} proved the following difficult theorem.

\begin{theorem}
  \label{thm:Szczechla-cycles}
  Let $G$ be a cycle on $n$ vertices. Then the sages win if and only if
  $n$ is divisible by $3$ or $n = 4$.
\end{theorem}

In the second section we give necessary definitions and describe 3-vertex and
9-vertex models which we use to describe strategies in matrix
language. The 9-vertex model is a~fundamental tool for the computational
approach to the problem. Many of our results were obtained by hard computer
calculations. We use the 9-vertex model for ``fast'', i.\,e.,  relatively simple,
checking strategies found by the computer. It is difficult to find a
strategy, but whenever it is found, we can prove that it is indeed winning strategy by simple
calculations with $9\times 9$ matrices. Unfortunately, only in several cases we
can replace these calculations with a ``logical'' proof.

In the third section we analyze variants of the game in which the sages get
some extra information (``a hint''). These variants of the game are of independent
interest. Moreover, hints can be used for constructing (difficult) strategies in
the general variant of the game.

In the forth section we prove the following main theorem that describes all
graphs where the sages win.

\setcounter{theorem}{19}

\begin{theorem}
  The sages lose on a connected graph $G$ if and only if $G$ is a tree or
  $G$ contains a~unique cycle $C_n$, where $n$ is not divisible by $3$, $n\geq
  5$.
\end{theorem}

\setcounter{theorem}{4}

\subsection{3-vertex and 9-vertex models}

Throughout the paper, let $G$ is an undirected graph without loops and
multiple edges. The sages occupy all vertices of $G$, we will
identify often a vertex with the~sage sitting in it. Let $H$ be the set of
hats colors, $H=\{0, 1, 2\}$. Denote by $C_n$ a cycle with $n$ vertices and by
$P_n$ a path with $n$ vertices. If we need to indicate the starting and
ending vertices of a path, we write $P_n(AB)$ or $P(AB)$ (for the description of
strategies involve oriented paths; in this case, we assume that the vertices
of the path are ordered in the direction from $A$ to $B$).

By the rules of the <<Hats>> game, the guess of each sage depends only on the hat colors 
of his neighbors. Therefore, a strategy of the sage in a vertex $v$ is
a function $f_v\colon H^{\deg(v)} \rightarrow H$. It is convenient to describe
the strategy $f_v$ by a $\deg(v)$-dimensional matrix. We put matrices defining strategies into 
square brackets; matrix entries have indices from $0$
to $2$ in each coordinate. For example, the sage that occupies a leaf vertex can
use the strategy ``I say what I see'' which corresponds to the matrix
$\begin{bmatrix}0&1&2\end{bmatrix}$. If a sage $v$ has two neighbours Left and
Right, then one of his possible strategies is
$$
\begin{bmatrix}
  0&1&2\\  
  1&1&1\\
  1&1&1
\end{bmatrix},
$$
which means that if the color of the Left's hat is $0$ then the sage $v$
calls out the color of the Right's hat. In all other cases, the sage $v$ calls out color $1$.
Usually, the sages are given names or numbers; in this
case, the matrix entries of the corresponding matrices are ordered
lexicographically. The list of all functions $\{f_v, \ v\in V\}$
will be called the \emph{(collective) strategy of the sages}.

Assume that the sages have fixed a collective strategy. \emph{A hat placement} is a function
${C\colon V\to H}$, i.\,e., it is defined if for each vertex $v$
of the graph $G$ a color $C(v)$ is chosen. We also denote the color of a vertex $v$
by $c_v$ (when this causes no ambiguity).
A hat placement is said to be \emph{losing}, or \emph{disproving}, if no 
sage makes a correct guess for this placement using the chosen strategy, 
i.\,e.\ for every vertex $v$ 
$$
f_v(C(u_1),\dots, C(u_k))\ne C(v),
$$
where $u_1, \dots, u_k$ is the list of neighbours of $v$. If $P$ is a path
in graph $G$ (or a cycle), then we say that a~hat placement on $P$  a \emph{disproving
chain} if no sage  in the inner vertices of $P$ makes a correct guess (in the case of the
cycle, no sage makes a correct guess).

\subsubsection{3-vertex model}

For a graph $G = \langle V, E \rangle$, let $3 * G$ be a graph with vertex set $V \times H$
in which two vertices $(u, i)$ $(v, j)$ are joined by an edge if and only
if $uv\in E$. We call the graph $3*G$ the \emph{3-vertex model} for graph $G$.  

It is convenient ot use the graph $3*G$ when considering hat placements on $G$. 
For a hat placement  $c = (c_1, c_2, \dots, c_\ell)$ on the vertex set $v=\{v_1, \dots,
v_\ell\}$ of graph $G$, we define the \emph{lifting} of~$c$ to the graph $3 * G$. 
By definition, it is a subgraph $G_c$ of $3 * G$ with vertex set $(v_i,
c_i)$ isomorphic to the graph $G$: vertices $(v_i, c_i), (v_j, c_j)$ are
joined by an edge if and only if $v_iv_j$ is an edge of $G$, see
Fig.~\ref{fig:part_of_3cn}. The 3-vertex model was the main tool in the proof of
theorem~\ref{thm:Szczechla-cycles} in~\cite{Szczechla}.

\setrisbox\vbox{\hsize=52pt\noindent
\begin{picture}(54,24)
  \linethickness{.1pt}
  \multiput(0,-2)(10,0){6}{\line(0,1){24}}
  \multiput(-2,0)(0,10){3}{\line(1,0){54}}
  \Thicklines
  \put(0,10){\vector(1,0){10}}
  \put(20,10){\vector(1,1){10}}
  \put(40,10){\vector(1,-1){10}}
\end{picture}}

\smallskip
\ris[\rismove(0pt,-1mm)]
Now describe another method of visualizing disproving chains, which uses Motzkin paths.
A \emph{Motzkin path} is an (oriented) lattice path with steps depicted on the
figure on the right.
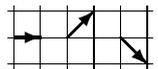
\endris

\setrisbox\vbox{\hsize=4.5cm
\begin{picture}(104,54)
\linethickness{.1pt}
\multiput(0,-2)(10,0){10}{\line(0,1){44}}
\multiput(-2,0)(0,10){5}{\line(1,0){94}}
\Thicklines
\path(10,10)(20,20)(30,20)(40,30)(60,10)(80,30)(90,30)
\put(0,-2){\vector(0,1){50}}
\put(-2,0){\vector(1,0){99}}
\footnotesize
\put(-8,-4){0}\put(-8,6){1}\put(-8,16){2}\put(-8,26){0}\put(-8,36){1}
\end{picture}}

Let the sages sitting it the vertices of a path $A_1A_2\dots A_n$  get hats $c_1$, $c_2$,
\dots, $c_n$ where ${c_i\in\{0, 1, 2\}}$ for every $i$. We convert this hat placement into a Motzkin path. 
The path starts at the point $(1, c_1)$ of the coordinate plane. 
Observe that in calculations modulo 3 any reminder can be ob- 
\ris
tained from any other one by adding $-1$, $0$ or $1$ in $\ZZ_3$. So, we
can define a path by applying the following rule: if the $i$th point of the
path has already been defined, then next line is $\rightarrow$ if $c_{i+1} = c_i$,
$\searrow$ if $c_{i+1}=c_i - 1 \bmod 3$, and $\nearrow$ if $c_{i+1} = c_i + 1
\bmod 3$. For example, the hat placement 122021200 (for $n=9$) corresponds to Motzkin
path shown in the figure on the right.
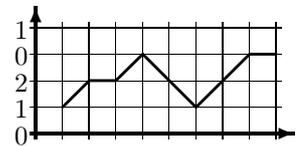
\endris

\subsubsection{9-vertex model}

Given a graph $G$, let $L(G)$ be the \emph{graph of edges} of $G$, i.\,e.\ 
every vertex of graph $L(G)$ corresponds to an edge of $G$, 
and two vertices of $L(G)$ are joined by an edge if and only if the 
corresponding edges in $G$ share a common vertex.

First, assume that $G$ is a path $P=P_n$ (or a cycle $C_n$). 
Consider the graph $G^{[9]}=L(3*G)$. Its vertices are pairs of neighbouring sages in
the graph $G$, with possible colors of hats indicated for each sage. We call 
$G^{[9]}$ the \emph{initial graph of 9-vertex model}. Every path $T$ in the graph $3 *
G$ uniquely determines a path $\widetilde{T}$ in the graph $G^{[9]}$, which will be called
\emph{the lifting} of $T$. So, a hat placement $c$ on the initial path $P$
gives rise to the lifting $P_c$ in the graph $3 * G$, which, in turn, gives rise to the lifting
$\widetilde{P_c}$ in the graph $G^{[9]}$. Thus, hat placements on the 
graphs $P$ and $C_n$ can be interpreted as paths in the graph $G^{[9]}$ (see
Fig.~\ref{fig:9vertex}).

\begin{figure}[t]
  \footnotesize
  \minipage{0.42\linewidth}
  \centering
  \begin{tikzpicture}   
      \begin{scope}[every node/.style={circle, draw, inner sep=1pt}]
        \foreach [count=\x] \name/\col in {A/0, B/2, C/1, D/2}    
        \foreach \y in {0}
           \node (\name\y) at (1.5 * \x, -1.5 * \y)[label=above: {\col}] {$\name$};  
      \end{scope}
      \foreach \l/\r in {A/B, B/C, C/D}
      \foreach \x in {0}
      \foreach \y in {0}
      \path (\l\x) edge  (\r\y);      
    \end{tikzpicture}
    \caption{A path $P$ and hat placement $c$ on it}\label{fig:chain_on_P}
\bigskip\bigskip\bigskip
    \begin{tikzpicture}   
      \footnotesize
      \begin{scope}[every node/.style={circle, draw, inner sep=1pt}]
        \foreach [count=\x] \name in {A, B, C, D}    
        \foreach \y in {0,1,2}    
        \node (\name\y) at (1.5 * \x, -1.5 * \y) {$\name_\y$};  
      \end{scope}
      \foreach \l/\r in {A/B, B/C, C/D}
      \foreach \x in {0, 1, 2}
      \foreach \y in {0, 1, 2}
      \path (\l\x) edge  (\r\y); 
      \path[ultra thick] (A0) edge  (B2);       
      \path[ultra thick] (B2) edge  (C1);       
      \path[ultra thick] (C1) edge  (D2);       
    \end{tikzpicture}
    \caption{The path $P_c$ is the lifting of  hat placement $c$ to graph $3*P$}\label{fig:part_of_3cn}
  \endminipage\hfill
  \minipage{0.57\linewidth}
  \centering
  \begin{tikzpicture}
    \footnotesize
    \begin{scope}[every node/.style={circle, draw, inner sep=1pt}]
      \foreach [count=\x] \n/\m in {A/B, B/C, C/D} {    
        \foreach \i in {0, 1, 2} {
          \foreach \j in {0, 1, 2} {
            \node (\n\i\m\j) at (4 * \x, - 3 * \i - \j) {$\n_\i\m_\j$};
          }
        }
      }            
    \end{scope}
    \foreach \na/\nb/\nc in {A/B/C, B/C/D}
    \foreach \a in {0, 1, 2}
    \foreach \b in {0, 1, 2}
    \foreach \c in {0, 1, 2}
    \path (\na\a\nb\b) edge (\nb\b\nc\c);
    \path[ultra thick] (A0B2) edge  (B2C1);       
    \path[ultra thick] (B2C1) edge  (C1D2);   
  \end{tikzpicture}
  
  $\underbrace{\qquad\quad}_{\text{Layer } AB}$\hskip 2.8cm
  $\underbrace{\qquad\quad}_{\text{Layer } BC}$\hskip 2.8cm
  $\underbrace{\qquad\quad}_{\text{Layer } CD}$
\caption{The path $\widetilde  P_c$ in the 9-vertex model $P^{[9]}$}
\label{fig:9vertex}
  \endminipage
\end{figure}

Now we define the graph $G^{[9]}$ in the general case. For an arbitrary graph $G$, 
let the \emph{initial graph of 9-vertex model} $G^{[9]}$  be the hypergraph of edges of the
graph $3*G$ defined as follows.

The vertices of the hypergraph $G^{[9]}$ correspond to the edges of graph $3 * G$. 
Thus, every edge in $G$ determines nine edges in $3 * G$ and nine vertices in $G^{[9]}$. 
The set of all vertices of $G^{[9]}$ that correspond to the same pair of
neighbouring sages i.\,e.\ to an edge of $G$ will be called a \emph{layer}.
For example, nine vertices $A_iB_j$ in the left column of Fig.~\ref{fig:9vertex}
form the layer corresponding to the pair of sages $A$ and $B$ from Fig.~\ref{fig:chain_on_P}.

Now we describe the hyperedges of $G^{[9]}$. Every vertex of the initial graph
$G$ determines exactly one hyperegde of graph $G^{[9]}$ in the following manner.
Let $u$ be an arbitrary vertex of $G$, let $v_1$, $v_2$,~\dots, $v_k$ be the
(unordered) set of all its neighbours and $(\alpha, \beta_1, \dots, \beta_k)$ be
an arbitrary set of colors. Then the pairs $(u,\alpha)$, $(v_1, \beta_1)$,
\dots, $(v_k, \beta_k)$ are vertices of the graph $3*G$, and the quadruples of the form
$\bigl((u,\alpha), (v_i, \beta_i)\bigr)$ are vertices of $G^{[9]}$. By
definition, the hyperedge determined by the vertex $u$ is the union of $(k+1)$ vertices
$$
\bigl((u,\alpha), (v_1, \beta_1)\bigr), \quad \bigl((u,\alpha), (v_2, \beta_2)\bigr), \quad\dots,
\bigl((u,\alpha), (v_k, \beta_k)\bigr)
$$
of hypergraph $G^{[9]}$.
For short, we will write this hyperedge as
\begin{equation}
  \label{eqn:multirebro}
  \Bigl((u,\alpha), \underbrace{(v_1, \beta_1),  \dots, (v_k, \beta_k)}_{\text{unordered list of pairs}}\Bigr).
\end{equation}

Now assume that the sages have chosen a strategy $f$. We remove some hyperedges from the  hypergraph
$G^{[9]}$. Namely, let a sage $u$ see a hat of color~$\beta_1$ on a sage $v_1$, 
a hat of color $\beta_2$ on a~sage $v_2$, \dots, a hat of color $\beta_k$ on sage $v_k$,
and let his guess according to the chosen strategy be a~color
$\alpha=f_u(\beta_1,\dots, \beta_k)$. Then we remove the hyperedge
$\Bigl((u,\alpha), (v_1, \beta_1),  \dots, (v_k, \beta_k)\Bigr)$ from~$G^{[9]}$. 
Perform this operation for each vertex $u$ and each set of colors
$(\beta_1,\dots,\beta_k)$. The resulting graph will be called the
\emph{9-vertex model of the graph $G$ for the strategy $f$} and denoted by $G_f^{[9]}$.

Given a disproving hat placement $c$ for a strategy $f$ on a graph $G$ (recall
that this means that a color $c_v$ is given for each vertex $v$ so that no sage
guesses the color of his own hat using the strategy $f$), we can define the~\emph{lifting} of the graph $G$ to
$G_f^{[9]}$ in the following way. First, we lift $G$ to $3 * G$, obtaining a
subgraph $G_c$ of the graph $3 * G$. Since we have taken a disproving hat
placement, for every vertex $v$ and all its neighbours none of the hyperedges of the 
form~\eqref{eqn:multirebro} in $G^{[9]}$ has been removed. Therefore we can lift
$G_c$ to~$G^{[9]}$: taking the edges of $G_c$ as vertices in $G^{[9]}$ and taking
all possible sets of vertices $G_c$ of the form~\eqref{eqn:multirebro} as
hyperedges yields a well-defined subhypergraph $\widetilde{G}$ of
$G_f^{[9]}$, which we call the \emph{lifting} of graph $G$ to $G_f^{[9]}$.

\begin{lemma}
  \label{lem:win9v}
  A strategy $f$ on graph $G$ is losing  if and only if there exists a hat
  placement that gives rise to a well-defined lifting of $G$ to $G_f^{[9]}$.
\end{lemma}

\begin{proof}
  We have already checked that every disproving hat placement can be lifted to $G_f^{[9]}$. 
	The inverse statement is also true: if the vertices of some
  subhypergraph $\bar{G}$ of $G_f^{[9]}$ project bijectively onto the edges of the graph $G$ 
	in such a way that every vertex of $G$  with all its outgoing edges corresponds to
  exactly one hyperedge in graph $G_f^{[9]}$, then subgraph $\bar{G}$ uniquely
  determines a~disproving hat placement the for strategy $f$ on the graph $G$. Moreover,
  in this case the graph $\bar{G}$ is the lifting of $G$ to $G_f^{[9]}$ constructed 
	from this hats placement.
\end{proof}

\medskip
E\,x\,a\,m\,p\,l\,e.  \,Consider the graph $G = P$ with the hat placement shown in
Fig.~\ref{fig:chain_on_P} (the vertices are ordered lexicographically). Let the
strategies of the sages $B$ and $C$ be chosen as follows:
\begin{equation}
  \label{eq:primer_strategii}
	f_B =
  \begin{bmatrix}
    1&0&0\\  
    1&2&1\\
    2&2&0
  \end{bmatrix},\qquad
  f_C =
  \begin{bmatrix}
    2&1&1\\
    0&0&0\\
    2&1&1
  \end{bmatrix}.
\end{equation}

Then the part of the graph $P_f^{[9]}$ corresponding to the fragment of the graph $L(3 *
P)$ in fig.~\ref{fig:part_of_3cn} is shown in fig.~\ref{fig:part_ofm9_f(cn)}.
Since the hat placement in Fig.~$\ref{fig:chain_on_P}$ is winning for the sages
(because the sage~$C$ guesses color of his hat correctly using the strategy $f$), this hat placement
does not determine the lifting of the initial graph $P$ to $P_f^{[9]}$, in particular,
we see that $P_f^{[9]}$ does not contain the edge $(B_2C_1)$--$(C_1D_2)$, which must
belong to such a lifting.

\begin{figure}[t]
  \minipage[b]{0.55\linewidth}
  \begin{tikzpicture}
    \footnotesize
    \begin{scope}[every node/.style={circle, draw, inner sep=1pt}]
      \foreach [count=\x] \n/\m in {A/B, B/C, C/D}
      \foreach \i in {0, 1, 2}
      \foreach \j in {0, 1, 2}
      \node (\n\i\m\j) at (4 * \x, - 3 * \i - \j) {$\n_\i\m_\j$};
    \end{scope}
    \def\fb{{
        {1,0,0},
        {1,2,1},
        {2,2,0}
      }}  
    \foreach \a in {0, 1, 2}
    \foreach \b in {0, 1, 2}
    \foreach \c in {0, 1, 2} {
      \pgfmathsetmacro{\nb}{\fb[\a][\c]}
      \ifthenelse{\lengthtest{\nb cm = \b cm} }{}{\path (A\a B\b) edge (B\b C\c);}
    }
    \def\fc{{
        {2,1,1},
        {0,0,0},
        {2,1,1}
      }}
    \foreach \a in {0, 1, 2}
    \foreach \b in {0, 1, 2}
    \foreach \c in {0, 1, 2}{
      \pgfmathsetmacro{\nb}{\fc[\a][\c]}
      \ifthenelse{\lengthtest{\nb pt = \b pt}}{}{
        \path (B\a C\b) edge (C\b D\c);
      }
    } 
  \end{tikzpicture}
  
  $\underbrace{\qquad\quad}_{\text{Layer } AB}$\hskip 2.6cm
  $\underbrace{\qquad\quad}_{\text{Layer } BC}$\hskip 2.6cm
  $\underbrace{\qquad\quad}_{\text{Layer } CD}$
  \caption{The graph $P_f^{[9]}$ for the startegy $f$ defined \goodbreak by formula \protect{\eqref{eq:primer_strategii}}}
  \label{fig:part_ofm9_f(cn)}
  \endminipage
  \hfil
  \minipage[b]{0.43\linewidth}
  \footnotesize
	\begin{align*}
    &&&{\arraycolsep=3pt\begin{matrix}
        \quad\!\!00&01&02&10&11&12&20&21&22
      \end{matrix}}
\\
    &{\renewcommand{\arraystretch}{1.2}\begin{matrix}
        00\\01\\02\\10\\11\\12\\20\\21\\22
      \end{matrix}}&  
      &\begin{pmatrix}
        1&0&0&0&0&0&0&0&0\\
        0&0&0&0&1&1&0&0&0\\
        0&0&0&0&0&0&1&1&1\\
        1&1&1&0&0&0&0&0&0\\
        0&0&0&0&1&0&0&0&0\\
        0&0&0&0&0&0&1&0&1\\
        1&1&0&0&0&0&0&0&0\\
        0&0&0&1&1&1&0&0&0\\
        0&0&0&0&0&0&0&0&1
      \end{pmatrix} 
  \end{align*} 
\caption{The matrix $M^{(B)}$ is the ajacency matrix for the layers $AB$ and $BC$ of
  the graph $P_f^{[9]}$ (the indices of rows and columns are obtained by erasing letters
  from the names of vertices in $P_f^{[9]}$)}
\label{fig:matricaMb} 
\endminipage
\end{figure}

\subsubsection{The adjacency tensor and counting the number of disproving hat placements}

Let $u$ be a vertex of graph $G$, and $v_1$, $v_2$, \dots, $v_k$ be the set of its
neighbours. Every edge $uv_i$ in~$G$ determines the layer $V_i$ consisting
of nine vertices of $G_f^{[9]}$ of the form $w^{(i)}_\gamma=(u,\alpha)(v_i,
\beta)$, where $\gamma=(\alpha, \beta)\in H\times H$ is a pair of colors. Let us
construct the \emph{adjacency tensor}, which describes in which case the point
$$
(w^{(1)}_{\gamma_1}, w^{(2)}_{\gamma_2}, \dots, w^{(k)}_{\gamma_k})\in
V_1\times V_2\times \dots\times V_k, \quad \text{ where} \  \gamma_i=(\alpha_i,
\beta_i)\in H\times H \ \text{for all $i$},
$$ 
is a hyperedge in $G_f^{[9]}$. By definition, let
$$
M^{(u)}_{\gamma_1 \gamma_2\dots \gamma_k }=
\begin{cases}
  1,& \vtop{\hsize = 11cm\sloppy\noindent if ${\alpha_1=\alpha_2=\dots =\alpha_k}$
    and  \hfill\hbox{\qquad\qquad} $\Bigl((u,\alpha_1), (v_1, \beta_1),  \dots,
    (v_k, \beta_k)\Bigr)$ is a hyperedge of $G_f^{[9]}$,}  \\
  0,& \text{otherwise} \\
\end{cases}
$$ 
For example, for $k=2$ this construction gives the ordinary adjacency matrix of the
bipartite subgraph formed by the two neighbouring layers of graph $G_f^{[9]}$. 
Figure~\ref{fig:matricaMb} demonstrates the matrix $M^{(B)}$ corresponding
to the vertex $B$ of the graph $P$ in Fig.~\ref{fig:chain_on_P} and the strategy $f$ defined
by~\eqref{eq:primer_strategii}. It describes the adjacency of the layers $AB$
and $BC$ of the graph $P_f^{[9]}$ shown in fig.~\ref{fig:part_ofm9_f(cn)}.

Observe an important technical detail. In the definition of the tensor
$M^{(u)}_{\gamma_1 \gamma_2\dots \gamma_k }$, the  coordinates  
$\gamma_i=(\alpha_i, \beta_i)$ of the parameters are listed in the order from the vertex $u$ to its
neighbours: $\alpha_i$ is a possible color of the vertex $u$, and $\beta_i$ is a
color of the  adjacent vertex $v_i$. But for matrices it is more convenient  to
parametrize the rows and columns with pairs of colors $\gamma=(\alpha,
\beta)\in H\times H$ as in fig.~\ref{fig:matricaMb}, i.\,e., with the coordinates of the first
parameter $\gamma_1=(\alpha_1, \beta_1)$ is written in the opposite order:
from a neighbouring vertex to the vertex $u$ (as in fig.~\ref{fig:matricaMb}, where
the order is from the vertex~$A$ to the vertex $B$). For an index $\gamma=(\alpha,
\beta)$ we write $\gamma^*=(\beta, \alpha)$.

If we have the adjacency tensors $M^{(u)}_{\gamma_1 \gamma_2\dots \gamma_k }$ and $M^{(v)}_{\gamma'_1
  \gamma'_2\dots \gamma'_\ell }$ for  two adjacent vertices $u$ and
$v$ of~$G$, then there are two indices  
$\gamma_j=(\alpha_u, \beta_v)$ and $\gamma'_j=(\beta_v, \alpha_u)=\gamma^*_i$,
that correspond to the edge~$uv$. We say that such indices are \emph{repeated}.

The main idea of using adjacency tensors can be illustrated as follows.
Consider the matrix $M^{(C)}$ describing the adjacency of the layers $BC$ and $CD$ 
in~fig.~\ref{fig:part_ofm9_f(cn)}. 
Contracting the product $M^{(B)}_{\gamma_1,\gamma_2}M^{(C)}_{\gamma^*_2,\gamma_3}$ 
on repeated index~$\gamma_2$, i.\,e.\ calculating the matrix product $M^{(B)}M^{(C)}$, we
obtain a~matrix whose elements count the number of paths in the graph
$G_f^{[9]}$ from the layer $AB$ to the layer $CD$. Every such path is, in fact, a
disproving chain for the graph $G=P$ shown in Fig.~\ref{fig:chain_on_P}, i.\,e.,
a~hat placement for which both sages do not guess the color of their hats correclty.

\smallskip
In the general case, contractions of adjacency tensors provide a
technique for counting the disproving hat placements. For instance, let us write the
products of the adjacency tensors over all vertices and all possible sets of indices:
$$
 M^{(u)}_{\gamma_1 \gamma_2\dots \gamma_k } M^{(v)}_{\gamma'_1 \gamma'_2\dots \gamma'_\ell }\dots 
$$
Assume that each bottom index is repeated twice. (By definition, 
repeated indices $\gamma$ are indices of vertices connected by edge that
describe the same hat arrangement on this edge.) Summing over all
repeated indices (every index runs over the set $H\times H$), we obtain the
\emph{full contraction} of the adjacency tensors.
$$
\sum M^{(u)}_{\gamma_1 \gamma_2\dots \gamma_k } M^{(v)}_{\gamma'_1 \gamma'_2\dots \gamma'_\ell }\dots
$$

\begin{theorem}
  \label{thm:tensor-convolution}
  Let $u$, $v$, \dots be the list of all vertices of a graph $G$. Let a
  strategy $f$ of sages be fixed, and assume that for every vertex $u$ the adjacency tensor
  $M^{(u)}_{\gamma_1 \gamma_2\dots \gamma_k }$ in hypergraph $G_f^{[9]}$ is
  given. Consider the full contraction of the adjacency tensors
  \begin{equation}
    \label{eqn:tensor-convolution}
    N=\sum M^{(u)}_{\gamma_1 \gamma_2\dots \gamma_k } M^{(v)}_{\gamma'_1 \gamma'_2\dots \gamma'_\ell }\dots \
  \end{equation}
  In this expression, every index is repeated twice, and the sum is over all 
	repeated indices (each repeated index runs over the set $H\times H$).
  Then $N$ equals the number of disproving hat placements for the strategy $f$.
\end{theorem}

\begin{proof}
  In fact, this is a tautology: disproving hat placements correspond bijectively
  to the nonzero summands (and each nonzero summand in this sum is equal to~1).  
  Indeed, every disproving hat placement determines a unique assignment of colors to the
  vertices of~$G$. By the definitions of a hyperedge in $G_f^{[9]}$ and the adjacency tensor, 
	for every vertex $u$ we have $ M^{(u)}_{\gamma_1\gamma_2\dots \gamma_k }=1$. Therefore, 
	every disproving hat placement determines a summand in which all the factors are equal to 1.
  And vice versa, if we assign colors to each vertex $u$ and its neighbours in such
  a way that $ M^{(u)}_{\gamma_1 \gamma_2\dots\gamma_k }=1$,	then this color assignment 
	determines a well-defined hat placement. Obviously this hat placement is disproving.
\end{proof}

Let us give an example of calculations in the 9-vertex model.
\medskip
\begin{lemma}
  \label{lem:primer_na-cycle}
  Let $G$ be a cycle $AS_1S_2\dots S_k$.
  Let the sage $A$ use the strategy ${\mathbf A}$, and all the other sages 
  use the strategy ${\mathbf S}$, where 
  \begin{equation}
    \label{eqn:2strategiesOnCycle}
    {\mathbf A} =
    \begin{bmatrix}
      0&1&1\\
      0&1&0\\
      1&1&1
    \end{bmatrix},
    \qquad
    {\mathbf S} =
    \begin{bmatrix}
      2&1&1\\
      0&0&0\\
      2&1&1
    \end{bmatrix}.
  \end{equation}
  Then there exist exactly $k + 1$ disproving chains in $G$, 
	and all of them assign a hat of color $2$ to the sage $A$. 
\end{lemma}

\begin{proof}
  The strategies ${\mathbf A}$ и ${\mathbf S}$ are written for the cyclic order of
  sages, i.\,e., we assume that $S_k$ is the left neighbour of~$A$ and
  $S_1$ is the right neighbour of~$A$; while $A$ is the left neighbour of~$S_1$ and 
	$S_2$ is the right neighbour of~$S_1$,
  and so on, $S_{k-1}$ is the left neighbour of~$S_k$ and $A$ is the right
  neighbour of $S_{k}$. The strategies ${\mathbf A}$ и ${\mathbf S}$ can be
  described in the 9-vertex model by the following matrices
  \[
    {\mathcal A}=\begin{pmatrix}
      0&1&1&0&0&0&0&0&0\\
      0&0&0&1&0&0&0&0&0\\
      0&0&0&0&0&0&1&1&1\\
      0&1&0&0&0&0&0&0&0\\
      0&0&0&1&0&1&0&0&0\\
      0&0&0&0&0&0&1&1&1\\
      1&1&1&0&0&0&0&0&0\\
      0&0&0&0&0&0&0&0&0\\
      0&0&0&0&0&0&1&1&1
    \end{pmatrix},\qquad
    {\mathcal S}=
    \begin{pmatrix}
      1&1&1&0&0&0&0&0&0\\
      0&0&0&1&0&0&0&0&0\\
      0&0&0&0&0&0&0&1&1\\
      0&0&0&0&0&0&0&0&0\\
      0&0&0&1&1&1&0&0&0\\
      0&0&0&0&0&0&1&1&1\\
      1&1&1&0&0&0&0&0&0\\
      0&0&0&1&0&0&0&0&0\\
      0&0&0&0&0&0&0&1&1
    \end{pmatrix}.
  \]
  The elements of the matrix ${\mathcal S}^k$ satisfy the linear recurrent equation. 
	One can check by induction that for $k \geq 3$  
  \[
    {\mathcal S}^k=\begin{pmatrix}
      1  &  1&  1&             k-1&0&0&0&             k-1&             k-1\\
      0  &  0&  0&               0&0&0&0&               0&               0\\
      0  &  0&  0&               1&0&0&0&               1&               1\\
      0  &  0&  0&               0&0&0&0&               0&               0\\
      k-2&k-2&k-2&\frac{k^2-3k+4}2&1&1&1&\frac{k^2-3k+4}2&\frac{k^2-3k+4}2\\
      1&  1&  1&             k-1&0&0&0&             k-1&             k-1\\
      1&  1&  1&             k-1&0&0&0&             k-1&             k-1\\
      0&  0&  0&               0&0&0&0&               0&               0\\
      0&  0&  0&               1&0&0&0&               1&               1
    \end{pmatrix}.\]
  Then
  \begin{equation}
    \label{eqn:MatrixASk}
    {\mathcal A}{\mathcal S}^k \  = \
    \begin{aligned}
      &&&{\footnotesize\arraycolsep=3pt\begin{matrix}
      \quad\, 00&\quad 01&\ 02&\quad 10&\quad\ 11&\ 12&\ 20&\quad\ 21&\qquad 22
      \end{matrix}}
\\
     &{\footnotesize\renewcommand{\arraystretch}{1.25}\begin{matrix}
     00\\ 01\\02\\10\\11\\12\\20\\21\\22
     \end{matrix}}&  
     &\begin{pmatrix}
      \boxed{0}&0&0&  1&0&0&0&  1&  1\\
      0&\boxed{0}&0&  0&0&0&0&  0&  0\\
      1&1&\underline{1}&  k&0&0&0&  k&  k\\
      0&0&0&\boxed{0}&0&0&0&  0&  0\\
      1&1&1&k-1&\boxed{0}&0&0&k-1&k-1\\
      1&1&1&  k&0&0&0&  k&  k\\
      1&1&1&  k&0&0&\boxed{0}&  k&  k\\
      0&0&0&  0&0&0&0&  \boxed{0}&  0\\
      1&1&1&  k&0&0&0&  k&  k
    \end{pmatrix}
  \end{aligned} 
\end{equation}
(for clarity, we index the rows and columns of the matrix with the parameters $\gamma = (\alpha, \beta)\in
H\times H$).

The indexing of the rows in the product ${\mathcal A}{\mathcal S}^k$ ``inherited'' from the 
matrix $A$: an index $(\alpha, \beta)$ of a row is a possible collection of hat
colors on the edge going from $A$ to the left, i.\,e., on the edge $S_kA$
(where $\alpha$ is a color of $S_k$, and $\beta$ is a color of $A$).
The indexing of the columns in ${\mathcal A}{\mathcal S}^k$ is inherited from the
matrix $S$ corresponding to the last sage~$S_k$. Here, an index 
$(\alpha, \beta)$  of a column is a possible collection of hat colors on the edge going from $S_k$ to
the right, i.\,e., on the edge $S_kA$, too.
Thus, the diagonal elements of the matrix ${\mathcal A}{\mathcal S}^k$ give the
number of cyclic disproving chains on the graph $C_{k+1}$ for the given strategies.

For example, 1 in the position $\gamma_1=(0,2)$, $\gamma_2=(0,2)$ on the diagonal 
(it is underlined in~\eqref{eqn:MatrixASk}) means that there exists only one chain that
assigns color $2$ to sage $A$ and color $0$ to the neighbouring sage $S_k$ (such
a chain must begin with the edge $(S_k)_0A_2$ in the sense of the 3-vertex model and, 
making a full cycle, end with the same edge). The other nonzero element on the diagonal is
number $k$ in the position $(22)(22)$; it also corresponds to chains that assign
color~$2$ to the sage $A$.

\end{proof}

\subsubsection{SAT applying}

Theorem \ref{thm:tensor-convolution} allows us to reduce the question about the
existence of winning strategies on graphs to the Boolean satisfiability problem (SAT). Note that
formula~\eqref{eqn:tensor-convolution} involves nonnegative integers, and
the result is compared with zero: if $N=0$, then the sages win, otherwise they lose.

Now we replace the arithmetic calculations with logical ones. Assume that we try to
find the strategy~$f$ for the sages. For each vertex $u$ and all its neighbours $v_1$,
$\dots$, $v_k$ in the graph~$G$, for each color $\alpha$, and for each list of colors
$\beta = (\beta_1, \dots, \beta_k)$, we
create a Boolean variable $m^{(u)}_{\alpha, \beta}$. Let the value of this variable be FALSE if
and only if the sage in vertex $u$ guesses his color correctly for this arrangement of
colors, i.\,e., if $f(\beta)=\alpha$ (using $\alpha$ and $\beta$ instead of
$\gamma$ we cut of the manifestly zero part of the adjacency tensor). For the
values of these variables to correspond to a real strategy, we require that for every
list of colors $\beta$ exactly one of the variables $m^{(u)}_{0, \beta}$,
$m^{(u)}_{1, \beta}$, $m^{(u)}_{2, \beta}$ takes the value FALSE.
For this, we require that the expression
\begin{equation}
  \label{eqn:log-strategy-restriction}
  \bigwedge\limits_{u, \beta} \Bigl(
    \bigl(
    \neg m^{(u)}_{0, \beta} \vee \neg m^{(u)}_{1, \beta} \vee \neg m^{(u)}_{2, \beta}
    \bigr) \wedge
    \bigl(
    m^{(u)}_{0, \beta} \vee m^{(u)}_{1, \beta}
    \bigr) \wedge
    \bigl(
    m^{(u)}_{0, \beta} \vee m^{(u)}_{2, \beta}
    \bigr) \wedge
    \bigl(
    m^{(u)}_{1, \beta} \vee m^{(u)}_{2, \beta}
    \bigr)
  \Bigr)
\end{equation}
is always true.
The logical analogue of formula~\eqref{eqn:tensor-convolution} 
\begin{equation}
  \label{eqn:log-win-strategy}
  \bigvee \Bigl( m^{(u)}_{\alpha, \beta_1, \dots, \beta_k }\wedge m^{(v)}_{\alpha', \beta'_1, \dots, \beta'_\ell}\wedge \dots\Bigr)
\end{equation}
has the value TRUE if for at least one hat placement at least one factor
$ m^{(u)}_{\alpha, \beta_1, \dots, \beta_k }$
in each conjunction is FALSE. In this case, there are no
disproving hat placements, and the strategy is winning.

Thus, the question about the existence of a winning strategy is reduced to the question
of whether one can simultaneously set the values of all the variables 
in such a way that both~\eqref{eqn:log-strategy-restriction} 
and the negation of~\eqref{eqn:log-win-strategy} are
TRUE. So, the original question  is reduced to the Boolean
satisfiability problem (SAT). It is in this way we have found all concrete strategies
presented in this paper. In calculations we use the  SAT solver Lingeling
(\verb"http://fmv.jku.at/lingeling/~") developed by a~team  headed by A.\,Biere (Institute for
Formal Models and Verification, Austria).

This approach is very efficient compared with the naive brute-force search. For example,
a~strategy on the cycle $C_n$ is determined for each vertex by a $3\times 3$ matrix as
in~\eqref{eqn:2strategiesOnCycle}. There exist $3^9$ variants for one matrix, and,
therefore, the brute-force approach in this case requires considering
$3^{9n}$ variants. For the SAT approach, we define $3^3=27$ variables
$m^{(u)}_{\alpha, \beta}$ for each vertex~$u$ and obtain a formula with $3^{3n}$
variables.

\subsection{The theory of hints}

\subsubsection{Hints}

In this section, we study different ``hints'', i.\,e., public ways to help the sages.
\emph{A hint} is a rule that restricts the range of possible configurations hat colors  or
changes the procedure of the test. The hint is reported to the sages before the session. It
means that everybody (the sages and the adversary) knows about the hint during the session.

Consider the following hints.
\begin{center}
\begin{tabular}{|p{2.5cm}|p{10cm}|}
      \hline
      Hint notation& The information given (to everybody) by the hint   \\ \hline
      $2A$     &The sage $A$ can make two guesses about his hat color\\ \hline
      $A-1$    &The sage $A$ gets a hat of one of two colors, and these colors are explicitly known\\ \hline
      $A=B$    &The sages $A$ and $B$ get hats of the same color\\ \hline
      $A\ne B$ &The sages $A$ and $B$ get hats of the different colors\\ \hline
      $A^*$    &The sage $A$ get hat of one of only two colors; 
			          these colors will be known (to him only) in the moment when he puts on the hat \\ \hline
\end{tabular}
\end{center}

For $i\in\{0,1,2\}$ we denote the hint $A-1$ by $A-\{i\}$ if the sage $A$ will not get a hat of color $i$.

Additionally, we agree that the adversary is always plays against the sages, 
i.\,e., if he can choose a hat placement or give a hint  
complying with rules in such a way that the sages lose, he does this.  

In the case of hint $A^*$, during the session the sages fix strategies for everybody except~$A$. 
For~$A$, they determine three strategies which he will use depending on the missing color. 

\subsubsection{Examples of concrete strategies in the ``Hats'' game with hints}

The fact that the sages can win with the hints $A\ne B$, $A=B$ or $A - 1$ 
is established in the following lemmas (for some concrete graphs). 
The strategies in this lemmas were found by computer search. 
Although finding strategies is a difficult problem, 
checking that these strategies win is almost trivial, and in some cases we can even 
suggest a ``logical proof''.
\begin{lemma}
    \label{lem:path_with_a=b}
    On the graph $P_n(AB)$ with $n\geq 2$, the sages win with the hint $A\ne B$ or $A=B$.
\end{lemma}

\begin{proof}
We will present a strategy for the sages on the graph $P_n(AB)$ with the hint $A \neq B$, $n\geq 2$.
Let the sages $A$ and $B$ use the strategy ${\mathbf A} = {\mathbf B} = \begin{bmatrix}1&2&0\end{bmatrix}$.
It is easy to see that for $n=2$, it is a winning strategy.
For $n>2$, let each of the other $n-2$ sages use the strategy
  \[
    {\mathbf S} =
    \begin{bmatrix}
      1&1&0\\
      1&2&2\\
      0&2&0
    \end{bmatrix}.
  \]

We will represent possible hat placements as Motzkin paths. The hint $A\ne B$
implies that a Motskin path that begins at the point $(1,c_1)$ cannot end at the
point $(n,c_1+3k)$ for any~$k$. Below we will use only the fact that such a path cannot end at the
point $(n,c_1)$.

It is easy to check that the sage $A$ acting accoring to the chosen strategy always guesses the
color of his hat correctly if the Motzkin path begins with a descending step $\searrow$.
Similarly, the sage~$B$ guesses the color of his hat correctly if the Motzkin path ends with
an ascending step $\nearrow$. As to the other sages, their strategies are quite
symmetric, and each of them guesses right if the path contains one of the following
fragments:
\begin{equation}
  \label{eq:3motzkin-details}
  \begin{picture}(24,14)(0,2)
    \linethickness{.1pt}
    \multiput(0,-2)(10,0){3}{\line(0,1){14}}
    \multiput(-2,0)(0,10){2}{\line(1,0){24}}
    \Thicklines
    \path(0,0)(10,10)(20,10)
  \end{picture}, \qquad
  \begin{picture}(24,14)(0,2)
    \linethickness{.1pt}
    \multiput(0,-2)(10,0){3}{\line(0,1){14}}
    \multiput(-2,0)(0,10){2}{\line(1,0){24}}
    \Thicklines
    \path(0,0)(10,10)(20,0)
  \end{picture} \quad\text{or}\quad
  \begin{picture}(24,14)(0,2)
    \linethickness{.1pt}
    \multiput(0,-2)(10,0){3}{\line(0,1){14}}
    \multiput(-2,0)(0,10){2}{\line(1,0){24}}
    \Thicklines
    \path(0,10)(10,10)(20,0)
  \end{picture}.
\end{equation}
To verify the last claim, we just need to check all matrix
elements of $\mathbf S$ one by one. We describe only one example: 
0 in the bottom left corner of $\mathbf S$ means that the chain of
hats colors 2, 0, 0 is winning for the sages. This chain corresponds to
the first fragment in~\eqref{eq:3motzkin-details}.

We claim that every Motzkin path from $(1,c_1)$ to $(n,c_n)$, where
$c_1\ne c_n$, either begins with a~descending step $\searrow$, or ends with
an ascending step~$\nearrow$, or has a fragment of the form~\eqref{eq:3motzkin-details}.
It is obvious.
Indded, if the path begins with an ascending step $\nearrow$  and does not
contain fragments~\eqref{eq:3motzkin-details}, then all the other steps of the path
must also be ascending, but then the last step is also ascending, as required. 
If the the path begins with the horizontal step $\rightarrow$, then either
all the other steps of the path are horizontal (but a strictly horizontal path is
not appropriate for us), or one of the steps is ascending, and then all the subsequent
steps are ascending, too.

Thus, in the case of the hint $A\ne B$, there are no disproving hat
placements for the presented strategy.

In case of the hint $A=B$, the sages, obviously, win for $n = 2$. For $n > 2$, 
denote the unique neighbour of the sage $B$ by $D$. Then the sages win if $B$ acts according to
the strategy ``I say what I see'' and the other sages play on the path $P(A,D)$ 
according to a winning strategy for the hint $A\neq D$.
\end{proof}

\begin{lemma}
  \label{thr:hint3_cycle}
  On the cycle $C_n$ with $n \geq 3$, the sages win with the hint $A-1$.
\end{lemma}

\begin{proof}
  Assume that the sages are given the hint $A-\{2\}$. Let the sage $A$ use the strategy 
	${\mathbf A}$, and all the other sages use the strategy ${\mathbf{ S}}$ from
  lemma~\ref{lem:primer_na-cycle}. It is proved in
  lemma~\ref{lem:primer_na-cycle} that all the disproving chains assign a hat of
  color $2$ to the sage $A$, but this is prohibited by the hint.
\end{proof}

Denote by $G^A$ a graph $G$ with a labeled vertex $A$. By the \emph{sum of labeled
graphs $G_1^A$ and $G_2^B$} we mean the graph obtained as the union of $G_1$ and $G_2$ 
in which vertices $A$ and $B$ are merged. The sum of labeled
graphs $G_1^A$ and $G_2^A$ will be called the 
\emph{sum of the graphs $G_1$ and $G_2$ with respect to the vertex $A$} 
and denoted by $G_1\gplus{A}G_2$.

The next lemma claims that if a sage $B$ in a leaf vertex of a graph $G'$
gets the hint $B-1$, then he can use a simple trick  to ``pass'' this hint to his unique
neighbour (denoted by~$A$), i.\,e., ``push'' the hint inside the graph. 
This allows us to construct a winning strategy of the sages on~$G'$ 
if there exists a winning strategy on the graph $G=G'\setminus\{B\}$ with hint $A-1$.

\begin{lemma}[<<pushing>> a hint]
  \label{lem:hint3_push}
  Let $G$ be a graph on which the sages win with the hint $A-1$ for some $A$. 
	Then the sages  win on the graph $G' = G\gplus{A} P_2(AB)$ with the hint $B-1$.
\end{lemma}

\begin{proof}
  Let $f$ be a winning strategy of the sages on the graph $G$ with the hint $A-\{2\}$. 
  We will construct a~strategy $F$ for the sages on the graph $G'$ with the hint $B-\{2\}$.
  Let the sage $B$ use the strategy $[ 0, 0, 1 ]$, and the sages from $G\setminus\{A,B\}$ use the
  strategy $f$. We construct a strategy for the sage $A$ as follows: 
	for every hat placement $c=(c_B, C_G)$ on the neighbours of~$A$ 
	(where $c_B$ is a hat color of $B$ and  $C_G$ is a list of hat colors of 
  all other neighbours of~$A$ in the graph $G$), set 
  $$
  F_A(c_B,C_G)=
  \begin{cases}
    2,& \text{if\ } c_B=0, \\ 
    f_A(C_G),& \text{if\ } c_B=1. 
  \end{cases}
  $$
  If $c_B=0$, then $A$ guesses right in the case $c_A = 2$ and $B$ guesses right in the case $c_A
  \ne 2$. If $c_B = 1$, then $B$ guesses right if $c_A = 2$. So, if the sage $A$ sees
  that $c_B=1$, then he must guarantee that  the sages win for all hat placements
  in which the hat color of~$A$ is not~2. This can be achieved if the sages on the
  graph $G$ use the strategy $f$.
\end{proof}

\subsubsection{The hint $A-1$ on the sum of graphs}

\begin{theorem}
  \label{thr:hint3_sum}
  Assume that on graphs $G_0^A$, $G_1^A$ and $G_2^A$ there exist
	winning strategies with the hint $A-1$. Then
	on the graph $G = G_0\gplus{A}G_1\gplus{A}G_2\gplus{A} P_2(AB)$ the sages win without any hints.
\end{theorem}

\begin{proof}
  Denote a winning strategy on the graph $G_i^A$ with the hint $A-\{i\}$ 
  by $f_i$. We will construct a winning strategy for the graph $G$. 
  Let the sage $B$ always call out the hat color of the sage $A$. 
  Let all sages from $G_i\setminus\{A\}$ act according to the strategy $f_i$.
  Finally, let the sage~$A$ act in the following way: if he sees that 
  the hat of $B$ has color $j$, then he uses the strategy $f_j$
  (we may think at this moment that only the sages on the subgraph $G_j^A$
  are really trying to guess the colors of their hats). 
  This is a winning strategy,
  because if $c_A=c_B$ then $B$ guesses right and if $c_A\neq c_B$ then somebody from
  subgraph $G_{c_B}^A$ guesses right.
\end{proof}

Theorem~\ref{thr:hint3_sum} allows to construct nontrivial examples of graphs where the sages win.

\begin{primer}
  For all positive integers $n$, $m$, $k$ (with $n, m, k\geq 3$) the sages win on
  the sum of graphs $C_n \gplus{A} C_m \gplus{A} C_k \gplus{A} P_2(AB)$
  (see Fig.~\ref{fig:trefoil}). This is obvious by lemma~\ref{thr:hint3_cycle}
  and theorem~\ref{thr:hint3_sum}.
\end{primer}

\begin{figure}[h]
  \centering
  \footnotesize
  \begin{tikzpicture}[scale=0.34]
    \begin{scope}[every node/.style={circle, draw, inner sep=1.2pt, fill, left}]
      \node (v0) at (0, 0) [label=right:A]{};
      \node (v1) at (-3, -2) {};
      \node (v2) at (-5, -1) {};
      \node (v3) at (-5, 1) {};
      \node (v4) at (-3, 2) {};
      \node (v5) at (3, -2) {};
      \node (v6) at (5, -1) {};
      \node (v7) at (5, 1) {};
      \node (v8) at (3, 2) {};
      \node (v9) at (-2, 3) {};
      \node (v10) at (-1, 5) {};
      \node (v11) at (1, 5) {};
      \node (v12) at (2, 3) {};
      \node (v13) at (0, -3) [label=right:B]{};
    \end{scope}
    \path (v0) edge (v1);
    \path (v1) edge (v2);
    \path (v2) edge (v3);
    \path (v3) edge (v4);
    \path (v4) edge (v0);
    \path (v0) edge (v5);
    \path (v5) edge (v6);
    \path (v6) edge (v7);
    \path (v7) edge (v8);
    \path (v8) edge (v0);
    \path (v0) edge (v9);
    \path (v9) edge (v10);
    \path (v10) edge (v11);
    \path (v11) edge (v12);
    \path (v12) edge (v0);
    \path (v0) edge (v13);
  \end{tikzpicture}
   \caption{The graph $C_5 \gplus{A} C_5 \gplus{A} C_5 \gplus{A} P_2(AB)$}
  \label{fig:trefoil}
\end{figure}

Another example of the similar technique is given in lemma~\ref{lem:2cycles}.

\subsubsection{Theorems on the ``Hats''  game with hints}

\smallbreak
\begin{theorem}
  For any graph $G$, the sages lose with the hint $2A$ if and only if $A$ is an isolated
  vertex of graph~$G$.
\end{theorem}

\begin{proof}
  It follows from the fact that the sages win with the hint $2A$ even on the graph
  $P_2(AB)$: the sage $B$ verifies the conjecture that the hat colors coincide,
  while the sage $A$ verifies the conjecture that hat colors are different.
\end{proof}

Obviously, the hint $A-1$ cannot give additional advantage to sages outside the
connected component of the vertex $A$.

\begin{theorem}
  \label{thr:hint-a2-win}
  Let $G$ be a connected graph and $A$ be a vertex og $G$. Then the sages lose on the graph
  $G$ with the hint $A-1$ if and only if $G$ is a tree.
\end{theorem}

\begin{proof}
  Let us prove that if $G$ is a tree, then the sages lose. Assume that they win, so
  they have a winning strategy on the graph $G^A$. Take three copies of the tree $G$. By
  theorem~\ref{thr:hint3_sum}, the  sages have a winning strategy in game without
  hints on the tree $G' = G \gplus{A} G \gplus{A} G \gplus{A} P_2(AB)$. But this is impossible by
  lemma~\ref{tree}, a contradiction.
  
  Let us prove that if $G$ contains a cycle, then the sages win. If the vertex $A$ belongs to the
  cycle, then the sages can use the strategy from lemma~\ref{thr:hint3_cycle}. 
  If it lies outside the cycle, then consider the subgraph 
  consisting of the cycle and a path that joins the vertex $A$ with this cycle. 
  The sages can ``push'' the hint from $A$ to the
  cycle by lemma~\ref{lem:hint3_push}, and win on the cycle using lemma~\ref{thr:hint3_cycle}.
\end{proof}

An interesting feature of the hints $A=B$ and $A\neq B$ is that these hints can help
the sages even if $A$ and $B$ belong to different connected components!

\smallbreak
\begin{theorem}
  \label{thm:A=B}
  For every graph $G$, the sages win with the hint $A=B$ if and only if

  1) the vertices $A$ and $B$ belong to the same connected component, or

  2) the connected component of the vertex $A$ or the component of the vertex $B$ is not a tree. 
\end{theorem}

\begin{proof}
  1) On the path $P(AB)$, the sages win if they use the strategy from 
  lemma~\ref{lem:path_with_a=b}.
  
  2) Let us prove that if, say, the connected component of $A$ contains a
  cycle $C_n$, then the sages win. Let the sage $B$ always call out color $2$, and all
  the sages from the connected component of $A$ use a winning strategy
  for the hint $A-\{2\}$, which exists by theorem~\ref{thr:hint-a2-win}. Then in the
  case $c_B = c_A = 2$ sage $B$ guesses right, and in the case $c_B = c_A \neq 2$
  somebody from the connected component of $A$ guesses right.
  
  If both components of the vertices $A$ and $B$ are trees, then, by
  theorem~\ref{thr:two-colors-in-tree}, there exist two colors $c^1_A$, $c^2_A$
  and two losing (for the game without hints) hat placements in the connected
  component of $A$ in which $A$ gets a hat of colors $c^1_A$ and $c^2_A$, respectively.
  There exist also analogous colors $c^1_B$ and $c^2_B$ for the vertex $B$. But
  $\left \{c^1_A, c^2_A\right \} \cap \left \{ c^1_B, c^2_B  \right \} \neq
  \varnothing$, since the number of possible hat colors in the game is 3. So, we can
  choose losing hat placements in both components of $A$ and $B$ in such a
  way that the sages $A$ and $B$ get hats of the same color. Then  none of the sages
  guesses right.
\end{proof}

\begin{theorem}
  \label{thm:Ane=B}
  1) The sages win with hint $A\neq B$ on every connected graph.

  2) Let sages $A$ and $B$ be in the different connected components of graph
  $G$.   Let both of these components be graphs on which the sages lose in the
  game without hints. Then the sages win on graph $G$ if and only if both
  connected components of vertices $A$ and $B$ are not trees.
\end{theorem}

\begin{proof}
  1) On the path $P(AB)$, the sages can win using the strategy from
  lemma~\ref{lem:path_with_a=b}.
  
  2) Let us prove that if both connected components of the vertices $A$ and $B$ contain
  a cycle, then the sages win. Indeed, by theorem~\ref{thr:hint-a2-win},
  in this case the sages have winning strategies on the connected
  component of~$A$ with the hint $A - \{2\}$ and 
  on the connected component of~$B$ with the hint $B-\{2\}$. 
  Let all the sages act according to these strategies.
  Because of the hint $A\ne B$, one of the sages $A$ or $B$ does indeed have a hat 
  not of color 2, which means that the sages win.

  Now assume that one of the connected components, say that of~$A$, is a tree. 
  Consider an arbitrary losing hat placement for the game without hints on the
  connected component of~$B$. Let the vertex $B$ have color $c_B$ in this
  game. By theorem~\ref{thr:two-colors-in-tree}, there exist two colors
  the sage $A$ can have in a disproving hat placement, and 
  at least one of  them does not coincide with~$c_B$. 
  Take a disproving hat placement in the connected component of~$A$ 
	that assigns this color to~$A$. We obtain a disproving hat
  placement for the whole graph that satisfies the restriction $A\ne B$, so
  the sages lose.
\end{proof}

Thus, theorem~\ref{thm:A=B} provide a complete analysis of the ``Hats'' game with the 
hint $A = B$. Theorem~\ref{thm:Ane=B} establishes an important result for the hint $A\ne B$, 
but for a comlete analysis of the game with this hint (on nonconnected
graphs) we need the result of our main theorem~\ref{main_thm} (its proof, of
course, does not use theorem~\ref{thm:Ane=B}).

\begin{theorem}
  \label{thm-HintAstar}
  The hint $A^*$ does not affect the result of the ``Hats'' game. 
\end{theorem}

\begin{proof}
  Assume that the sages win with the hint $A^*$. For all sages except $A$, fix their
  strategies in the game with the hint $A^*$; we will construct a 
  strategy of the sage $A$ in such way that the sages win without hints.

  Assume that the adversary gives to $A$ a hat of color $x$, 
  and after that there exists a~hat placement in which $A$ gets a hat of color $x$,
  his neighbours get hats of colors $u$, $v$, $w$ \dots, 
  the other sages also get hats of some colors, and nobody (except $A$) guesses right.
  In this case, we want the sage $A$ to guess the color of his hat correctly,
  i.\,e., his strategy must satisfy the requirement $f_A(u, v, w, \dots) = x$. 

  These requirements for different hat placements do not contradict each
  other. Indeed, if there exists another hat placement 
  where the neighbours still have colors $u$, $v$, $w$, \dots and 
  the sage $A$ gets another color $y$, then the
  sages can not win with the hint $A^*$, because the adversary can inform $A$ that
  he has a hat of color $x$ or $y$ and then choose one of these two hat
  placements for which sage $A$ does not guess his color correctly.
\end{proof}

\begin{corollary}\label{sledst:add-leaf}
  Let $AB$ be an edge of a graph $G$ with $B$ being a leaf vertex. If the sages lose
  (without hints) on the graph $G\setminus \{B\}$, then they lose also on the graph $G$.
\end{corollary}

\begin{proof}
  Assume that the sages have found a winning strategy on the graph $G$ and  
  the strategy of the sage $B$ is $\begin{bmatrix}c_0&c_1&c_2\end{bmatrix}$. 
  Observe that $c_0$, $c_1$, $c_2$ are the three different colors. 
  Indeed, if this set does not contain, say, color~2, then consider 
  hat placements where $B$ gets a hat of color~2. In fact, 
  the strategies of all sages except $A$ work on the graph  $G\setminus \{B\}$.  
  When the sage $A$ sees a hat of color $2$ on the sage $B$, then his strategy 
  on the graph $G\setminus \{B\}$ is also well defined. 
  Choose a disproving hat placement on $G\setminus \{B\}$ 
  for this collective strategy and add a~hat of color~2 on the sage~$B$. 
  We obtain a disproving hat placement on the graph $G$.  

  So, the strategy $f_B$ of the sage $B$ must be designed in such way that 
  each of the three colors can be his guess 
  for a suitable color of his neighbor's hat. Therefore, if he was given 
  a hat of color~$c^*$, there is a uniquely determined color $c^{**}$ 
  such that $f_B(c^{**}) = c^*$. Then for the hat placements
  that assign to $A$ a hat of color $c^{**}$ (and only for them), 
	the sage $B$ guesses right. Hence for hat placements in which 
  the hat color of $A$ is not $c^{**}$, somebody from the graph $G\setminus\{B\}$ 
	must guess right. But this means that the sages win with the hint $A^*$ on the 
	graph $G\setminus\{B\}$ (this hint will be given at the moment when $A$ sees
  the color $c^*$ of $B$'s hat). But this is impossible 
  by theorem~\ref{thm-HintAstar}: if the sages lose on the graph 
  $G\setminus\{B\}$ without hints, 
  then they also lose on this graph with the hint $A^*$. A contradiction. 
\end{proof}

\subsection{Analysis of the game on arbitrary graphs}

Now, we deal with the general variant of the game for an arbitrary graph $G$ 
without hints. First, consider several special cases.

\begin{lemma}
  \label{lem:2cycles}
  Let $G$ be a connected graph that contains two cycles without common vertices.
  Then the sages win.
\end{lemma}

\begin{proof}
  Choose two neighboring vertices $A$ and $B$ on a path that connects the cycles
  and consider disjoint connected subgraphs $G_A$ and $G_B$ such that the
  vertex $A$ and the first cycle are contained in $G_A$, the vertex $B$ and the second
  cycle are contained in $G_B$. By theorem~\ref{thr:hint-a2-win}, the sages win on the 
  graph $G_A$ with the hint $A-1$ and also the sages win on the graph $G_B$ with the hint
  $B-1$.

  Wew will describe a winning strategy for the sages on $G$.

  Let the sages on $G_A \setminus \{A\}$ use the strategy for the game with the hint
  $A-\{2\}$, and the sages on $G_B \setminus \{B\}$ use the strategy for the
  game with the hint $B-\{2\}$. Let the strategy of $A$ be as follows: if $c_B =
  2$ then $A$ makes his guess according to the strategy on $G_A$ with the hint
  $A-\{2\}$, otherwise (when $c_B \neq 2$) he says ``2''. Let the strategy of
  $B$ be as follows: if $c_A = 2$ then $B$ says ``$2$'', and if  $c_A \neq 2$
  then $B$ uses the strategy for $G_B$ with the hint $B-\{2\}$.

  In the following table we consider all possible hat placements for $A$ and
  $B$ and demonstrate who wins.
  \begin{center}
    \begin{tabular}{|c|c|c|c|c|c|c|c|c|c|}
      \hline
      $c_A$      &0&0&0&1&1&1&2&2&2\\ \hline
			$c_B$      &0&1&2&0&1&2&0&1&2\\ \hline
			Who wins&$G_B$&$G_B$&$G_A$&$G_B$&$G_B$&$G_A$&$A$&$A$&$B$\\ \hline
    \end{tabular}
  \end{center}
\end{proof}

\begin{lemma}
  \label{lem:2cyclesWithCommonPoint}
  Let $G$ be a graph containing two cycles that have only one common vertex.
  Then the sages win.
\end{lemma}

\begin{proof}
  We have anly a technical proof. The strategies were found by a computer search, the
  proof that they are winning was performed in the computer algebra system
  Maple. It suffices to consider the case $G=C_{k+1} \gplus{A} C_{m+1}$.

  We present the following almost stationary strategy. Let the sage $A$ use the strategy $\mathbf A$,
  the sages $B_{k+1}$ and~$D_1$ use the strategy $\mathbf T$, 
  and all the other sages use  the strategy $\mathbf S$. 
  (see fig.~\ref{fig:2cyclesl}, we mark the sages with the letters denoting
  their strategies).

\begin{figure}[b]
 \minipage{0.45\textwidth}
  \footnotesize
  \begin{tikzpicture}[
    thick,
    scale=0.8,
      every node/.style={circle, draw,inner sep=2pt,fill},
      every edge/.style={draw}
      ]      
      \node (v0) at (0, 0) [label=right:$A$]{};
      \node (v1) at (-1, 2) [label=left:T,label=above:$D_1$]{};
      \node (v2) at (-3, 1) [label=left:S,label=above:$D_2$]{};
      \node (v3) at (-3, -1) [label=left:S]{};
      \node (v4) at (-1, -2) [label=left:S,label=below:$D_{m+1}$]{};
      \node (v5) at (1, 1.5) [label=right:T,label=above:$B_{k+1}$]{};
      \node (v6) at (3, 2.5) [label=right:S,label=above:$B_{k}$]{};
      \node (v7) at (5, 1) [label=right:S]{};
      \node (v8) at (5, -1) [label=right:S]{};
      \node (v9) at (3, -2.5) [label=right:S,label=below:$B_2$]{};
      \node (v10) at (1, -1.5) [label=right:S,label=below:$B_1$]{};
      \path (v0) edge (v1);
      \path (v1) edge (v2);
      \path (v2) edge (v3);
      \path (v3) edge (v4);
      \path (v4) edge (v0);
      \path (v0) edge (v5);
      \path (v5) edge (v6);
      \path (v6) edge (v7);
      \path (v7) edge (v8);
      \path (v8) edge (v9);
      \path (v9) edge (v10);
      \path (v10) edge (v0);
  \end{tikzpicture}
  \endminipage
  \minipage{0.55\textwidth}
  \footnotesize
  \[
    \mathbf A=\begin{bmatrix}
      \begin{bmatrix}
        2&0&2\\
        2&0&2\\
        0&0&2        
      \end{bmatrix}
      &
      \begin{bmatrix}
        2&0&2\\
        2&0&2\\
        0&0&2        
      \end{bmatrix}
      &
      \begin{bmatrix}
        2&0&2\\
        2&0&2\\
        0&0&2        
      \end{bmatrix}\\\\
      \begin{bmatrix}
        1&1&1\\
        1&1&1\\
        1&1&1        
      \end{bmatrix}
      &
      \begin{bmatrix}
        2&0&2\\
        2&0&2\\
        0&0&2        
      \end{bmatrix}
      &
      \begin{bmatrix}
        2&0&2\\
        2&0&2\\
        0&0&2        
      \end{bmatrix}\\\\
      \begin{bmatrix}
        1&1&1\\
        1&1&1\\
        1&1&1        
      \end{bmatrix}
      &
      \begin{bmatrix}
        1&1&1\\
        1&1&1\\
        1&1&1        
      \end{bmatrix}
      &
      \begin{bmatrix}
        2&0&2\\
        2&0&2\\
        0&0&2        
      \end{bmatrix}    
    \end{bmatrix},\]
  \[
    \mathbf T=
    \begin{bmatrix}
      0&2&1\\
      0&1&0\\
      2&2&1
    \end{bmatrix}, \qquad
    \mathbf S=
    \begin{bmatrix}
     2&2&1\\ 
     0&0&1\\ 
     2&2&1 
    \end{bmatrix}\]
  \endminipage
  \caption{The strategies for  $C_{k+1} \gplus{A} C_{m+1}$  ($k=5$, $m=3$)}
  \label{fig:2cyclesl}
\end{figure}

The following matrices describe the strategies $\mathbf T$ and $\mathbf S$ in the
9-vertex model
\[
    {\mathcal T}=\begin{pmatrix}
      0&1&1&0&0&0&0&0&0\\
      0&0&0&1&0&0&0&0&0\\
      0&0&0&0&0&0&1&1&1\\
      0&1&0&0&0&0&0&0&0\\
      0&0&0&1&0&1&0&0&0\\
      0&0&0&0&0&0&1&1&1\\
      1&1&1&0&0&0&0&0&0\\
      0&0&0&0&0&0&0&0&0\\
      0&0&0&0&0&0&1&1&1
    \end{pmatrix},\qquad
    {\mathcal S}=
    \begin{pmatrix}
      1&1&1&0&0&0&0&0&0\\
      0&0&0&1&0&0&0&0&0\\
      0&0&0&0&0&0&0&1&1\\
      0&0&0&0&0&0&0&0&0\\
      0&0&0&1&1&1&0&0&0\\
      0&0&0&0&0&0&1&1&1\\
      1&1&1&0&0&0&0&0&0\\
      0&0&0&1&0&0&0&0&0\\
      0&0&0&0&0&0&0&1&1
    \end{pmatrix}.
\]
By a routine induction or using a computer, one can calculate
\begin{gather}
  \label{eqn:matrixSkT}
  {\mathcal{S}^k\mathcal{T}}=
  \begin{pmatrix}
    0&k&1&k&1&k-1&k-1&0&\frac{k^2-k+2}2\\
    0&1&0&1&0&  1&  1&0&  k-1\\
    0&0&0&{\mathbf0}&0&  0&  0&1&    1\\
    {\mathbf0}&\boxed{0}&0&0&\boxed{0}&  0&  \boxed{0}&0&    1\\
    0&1&0&1&\boxed{0}&  1&  1&\boxed{0}&  k-1\\
    0&k&1&k&1&k-1&k-1&\boxed{0}&\frac{k^2-k+2}2\\
    {\mathbf0}&k&1&k&1&k-1&k-1&0&\frac{k^2-k+2}2\\
    0&1&{\mathbf0}&1&0&  1&  1&0&  k-1\\
    0&0&{\mathbf0}&0&{\mathbf0}&  0&  0&0&    1
  \end{pmatrix}, \allowdisplaybreaks
  \\
  {\mathcal{T}\mathcal{S}^m}=
  \begin{pmatrix}
    0&0&  1&  1&  1&0&0&0&   m-1\\
    0&0&  1&  1&  1&0&0&0&   m-1\\
    1&1&m-1&m-1&m-1&0&0&0&\frac{m^2-3m+6}2 \\
    0&0&  1&  1&  1&0&0&0&   m-2\\
    1&1&m-1&m-1&m-1&0&0&0&\frac{m^2-3m+6}2\\
    1&1&  m&  m&  m&0&0&0&\frac{m^2-m+2}2\\
    1&1&  m&  m&  m&0&0&0&\frac{m^2-m+2}2\\
    0&0&  1&  1&  1&0&0&0&   m-1\\
    0&0&  0&  0&  0&0&0&0&   1
  \end{pmatrix}.\notag
\end{gather}
To find the number of disproving hat placements, it remains
to calculate the contraction of the tensor~$\mathcal{A}$ corresponding to the strategy
$\mathbf{A}$ with the matrices $\mathcal{S}^k\mathcal{T}$ and
${\mathcal{T}\mathcal{S}^m}$. It is equal to $0$.

The last claim can be checked directly without computer. For this, observe
that if on the left cycle $AD_1D_2\dots D_m$ the sage $D_1$ uses the strategy $\mathbf
T$, the sages $D_2$, \dots, $D_m$ use the same strategy $\mathbf S$, and the sage $A$
use the strategy
$$ 
{\mathbf A}_{(00)(00)}=
\begin{bmatrix}
  2&0&2\\ 
  2&0&2\\ 
  0&0&2 
\end{bmatrix},
$$
then, in fact, the sages on this cycle act according to a winning strategy for the 
hint ${A-\{1\}}$. Indeed, the only two nonvanishing diagonal elements of the
product ${\mathcal A'}\mathcal{T}\mathcal{S}^m$ correspond to disproving hat
placements that assign color 1 to the hat of $A$ (here ${\mathcal A'}$ is the
matrix for the strategy $({\mathbf A}_{(00)(00)})^\text{t}$ in the 9-vertex model, the
transpose is needed to correctly take into account the repeated index). Thus,
depending on the hat colors of the sages $B_1$ and $B_{k+1}$, the sage $A$ either always
says ``1'' or acts on the second cycle according to the strategy $A-\{1\}$. It
remains to check that if no sage on the second cycle (including~$A$) guesses right,
then some sage on the first cycle does. These cases are listed in the
following table:
\begin{center}
  \begin{tabular}{|c|c|c|c|c|c|c||c|c|c|c|c|c|}
    \hline
    $c_A$         &0&0&0&2&2&2& 1&1&1&1&1&1\\ \hline
    $c_{B_1}$   &1&2&2&1&2&2& 0&1&2&0&0&1\\ \hline
    $c_{B_k}$   &0&0&1&0&0&1& 0&1&2&1&2&2\\ \hline
  \end{tabular}
\end{center}
The number of disproving chains for the path $B_1\dots B_{k+1}$ in each of these
cases is ``encoded'' by a~suitable matrix element of
${\mathcal{S}^k\mathcal{T}}$, see.~\eqref{eqn:matrixSkT}, and for each case this
matrix element vanishes. For example, the first column of the table corresponds to zero
element $({\mathcal{S}^k\mathcal{T}})_{(01)(00)}$ (we use the indexing rule shown
in Fig.~\ref{fig:matricaMb}). For a convenience, we bold the zeroes 
that correspond to the left half of the table, and frames the
zeroes for the right half. So, in all the cases under consideration 
disproving chains do not exist. Hence the sages win.
\end{proof}

\begin{lemma}
  \label{lem:3Paths}
  Let $G$ be a connected graph that contain vertices $A$ and $B$
  joined by 3 nonintersecting paths. Then the sages win. 
\end{lemma}

\begin{proof}
We have only a technical proof. 
Consider two cases: 

1) each of the 3 paths contains at least one (inner) vertex; a strategy
for this case is presented in Fig.~\ref{pic:3path333}, 

2) one of the paths is degenerate and consists of the edge $AB$;
a~strategy for this case is presented in Fig.~\ref{pic:3path-DirectEdge}.

The three paths are shown in the left parts of both figures, each vertex is
marked by the letter denoting the strategy, the strategies themselves are presented in the right parts of
the figures. We assume that the neighboring vertices for $A$ and $B$ are ordered
from top to bottom. Thus, the  matrix $\mathbf{A}[0]$ describes the actions of the sage $A$ 
if his topmost neighbor has a hat of color~0. We have checked that the strategies
are winning by calculations in the 9-vertex model analogous to the
calculations of lemma~\ref{lem:2cyclesWithCommonPoint}.

\begin{figure}[b]                                       
  \centering
  \minipage{0.4\textwidth}
  \footnotesize
  \begin{tikzpicture}[
    scale=1.2,
      every node/.style={circle, draw,inner sep=2pt,fill},
      every edge/.style={draw}
      ]
      \node (A) at (0, 0) [label=left:A]{};
      \node (v0) at (1, 1) [label=above:S]{};
      \node (v1) at (2, 1) [label=above:S]{};
      \node (v2) at (3, 1) [label=above:S]{};
      \node (v4) at (1, 0) [label=above:S]{};
      \node (v5) at (2, 0) [label=above:S]{};
      \node (v6) at (3, 0) [label=above:S]{};
      \node (v8) at (1, -1) [label=above:W]{};
      \node (v9) at (2, -1) [label=above:S']{};
      \node (v10) at (3, -1) [label=above:S']{};
      \node (B) at (4, 0) [label=right:B]{};
      \path (A) edge (v0);
      \path (v0) edge (v1);
      \path (v1) edge (v2);
      \path (v2) edge (B);
      \path (A) edge (v4);
      \path (v4) edge (v5);
      \path (v5) edge (v6);
      \path (v6) edge (B);
      \path (A) edge (v8);
      \path (v8) edge (v9);
      \path (v9) edge (v10);
      \path (v10) edge (B);
  \end{tikzpicture}
  \endminipage
  \minipage{0.6\textwidth}
  \footnotesize
 \begin{align*}
    \mathbf{A}[0]&=\begin{bmatrix} 0&0&0\\1&1&1\\1&1&1 \end{bmatrix},&
    \mathbf{A}[1]&=\begin{bmatrix} 0&0&0\\2&1&2\\0&0&2 \end{bmatrix},&
    \mathbf{A}[2]&=\begin{bmatrix} 0&0&0\\2&2&2\\0&0&2 \end{bmatrix},
  \\
    \mathbf{B}[0]&=\begin{bmatrix} 1&2&2\\0&2&0\\0&0&0 \end{bmatrix},&
    \mathbf{B}[1]&=\begin{bmatrix} 1&1&1\\1&2&2\\1&2&2 \end{bmatrix},&
    \mathbf{B}[2]&=\begin{bmatrix} 1&1&1\\0&2&0\\0&2&0 \end{bmatrix},
  \\
    \mathbf W&=   \begin{bmatrix} 2&2&2\\1&2&0\\1&1&0 \end{bmatrix},&
    \mathbf S&=   \begin{bmatrix} 1&1&1\\0&2&0\\0&2&0 \end{bmatrix},&
    \mathbf {S'}&=\begin{bmatrix} 1&1&0\\1&2&2\\1&2&0 \end{bmatrix}.
   \end{align*}
  \endminipage
  
  \caption{The upper and middle paths have at least one inner vertex,
  the lower path has at least two vertices. The strategy $S$ 
  is written for the ordering of vertices from left to right for the upper path, 
  and from right to left for the middle path. }
  \label{pic:3path333}
\end{figure}

\begin{figure}[t]                                       
  \centering
  \minipage{0.4\textwidth}
  \footnotesize
  \begin{tikzpicture}[
    scale=1.2,
      every node/.style={circle, draw,inner sep=2pt,fill},
      every edge/.style={draw}
      ]
      \node (A) at (0, 1) [label=left:A]{};
      \node (v4) at (1, 0) [label=above:S]{};
      \node (v5) at (2, 0) [label=above:S]{};
      \node (v6) at (3, 0) [label=above:S]{};
      \node (v8) at (1, -1) [label=above:S]{};
      \node (v9) at (2, -1) [label=above:S]{};
      \node (v10) at (3, -1) [label=above:S]{};
      \node (B) at (4, 1) [label=right:B]{};
      \path (A) edge (B);
      \path (A) edge (v4);
      \path (v4) edge (v5);
      \path (v5) edge (v6);
      \path (v6) edge (B);
      \path (A) edge (v8);
      \path (v8) edge (v9);
      \path (v9) edge (v10);
      \path (v10) edge (B);
  \end{tikzpicture}
  \endminipage
  \minipage{0.6\textwidth}
  \footnotesize
 \begin{align*}
    \mathbf{A}[0]&=\begin{bmatrix} 1&1&0\\2&0&2\\1&1&0 \end{bmatrix},&
    \mathbf{A}[1]&=\begin{bmatrix} 0&0&0\\2&2&2\\0&0&0 \end{bmatrix},&
    \mathbf{A}[2]&=\begin{bmatrix} 1&1&0\\2&2&2\\1&1&2 \end{bmatrix},
  \\
    \mathbf{B}[0]&=\begin{bmatrix} 0&2&1\\1&2&1\\0&2&2 \end{bmatrix},&
    \mathbf{B}[1]&=\begin{bmatrix} 1&2&1\\1&2&1\\1&2&1 \end{bmatrix},&
    \mathbf{B}[2]&=\begin{bmatrix} 1&2&1\\1&0&1\\0&2&0 \end{bmatrix},
  \\
    &&\mathbf S&=\begin{bmatrix} 1&2&1\\1&1&1\\0&2&0 \end{bmatrix}.
    \end{align*}
  \endminipage
  
  \caption{The strategy for the case where one of the paths 
  is the edge $AB$ itself, and other paths are arbitrary.
  The strategy $S$ is written for the ordering of vertices from left to right for the middle path, 
  and from right to left for the lower path.}
  \label{pic:3path-DirectEdge}
\end{figure}

Remarkably, these strategies are universal: they work on paths $AB$ of
arbitrary length with only two exceptions. For the strategy in
Fig.~~\ref{pic:3path-DirectEdge}, we must assume that one of the paths $P(AB)$
(we may think that it is the lower one) contains at least two vertices. Besides,
the strategy does not work if two of the paths (the middle and the lower one) both
contain exactly one inner vertex. In these exceptional cases the sages win by
theorem~\ref{thm:Szczechla-cycles}.

This universality is explained by the fact that, as in formula~\eqref{eqn:MatrixASk} 
of lemma~\ref{lem:primer_na-cycle}, the distribution of zero and nonzero elements 
in the powers ${\mathcal S}^k$ stabilizes for large $k$. This rule can be violated only for small $k$.
\end{proof}

Now we will prove the main theorem.

\begin{theorem}
  \label{main_thm}
  The sages lose on a connected graph $G$ if and only if either $G$ is a tree or
  it contains a unique cycle $C_n$ where $n$ is not divisible by $3$, $n\geq 5$.
\end{theorem}

\begin{proof}
If $G$ is a tree, then  the sages lose by theorem~\ref{tree}. If $G$ is a
cycle $C_n$ where $n$ is not divisible by 3, $n\geq 5$, then the sages lose by
theorem~\ref{thm:Szczechla-cycles}. Each graph containing a unique cycle can be
obtained from the cycle by successively applying the operation ``appending a
leaf''. By corollary~\ref{sledst:add-leaf} this operation does not affect the
property ``the sages lose''. Thus, the sages lose for all graphs mentioned in the statement
of the theorem.

Now we will prove that the sages win for all the other graphs. If $G$ contains a
unique cycle $C_n$ where $n=4$ or $n$ is divisible by $3$, then the sages win
by theorem~\ref{thm:Szczechla-cycles}.

Let $G$ contain at least two cycles. If it contains two nonintersecting
cycles, then the sages win by lemma~\ref{lem:2cycles}. If it contains two cycles
with one common vertex, then the sages win by
lemma~\ref{lem:2cyclesWithCommonPoint}. In all the other graphs containing two
cycles, one can choose two vertices $A$ и $B$ joined with 3 pairwise disjoint 
paths (and hence the sages win by lemma~\ref{lem:3Paths}).
Indeed, let the graph contain two cycles with (at least) two common vertices.
Consider two cycles with minimum sum of lengths. The two common vertices split
the cycles onto ``arcs''. Choose the shortest arc $AB$. Then this arc $AB$ and
the two paths from $A$ to $B$ along the complementary arc and along the second
cycle form a set of three nonintersecting paths.
\end{proof}

\end{document}